\documentclass[11pt]{article}
\usepackage{a4wide}
\usepackage{amsmath}
\usepackage{amssymb}
\usepackage{amsthm}
\usepackage{xcolor}

\usepackage{parskip}
\usepackage{hyperref}
\usepackage{parskip}
\usepackage{setspace}

\usepackage[margin=1in]{geometry}
\usepackage{amsmath,amsthm,amssymb}
\usepackage{hyperref}

\numberwithin{equation}{section}

\newtheorem{theorem}{Theorem}[section]
\newtheorem{lemma}[theorem]{Lemma}
\newtheorem{question}[theorem]{Question}
\newtheorem{claim}[theorem]{Claim}

\theoremstyle{remark}
\newtheorem{remark}[theorem]{Remark}

\begin{document}

\title{Raimi's theorem for the $n$-dimensional torus}
\author{
\textsc{Hunseok Kang}$^{1}$%
\thanks{%
$^{1}$College of Engineering and Technology, American University of the Middle East, Kuwait.\newline
\hspace*{0.6cm} Email: \texttt{hunseok.kang@aum.edu.kw}}%
\and
\textsc{Doowon Koh}$^{2}$%
\thanks{%
$^{2}$Chungbuk National University, Cheongju, Chungbuk 28644, Korea.
Email: \texttt{koh131@chungbuk.ac.kr}}%
\and
\textsc{Dung The Tran}$^{3}$%
\thanks{%
$^{3}$Vietnam National University, University of Science, Hanoi, Vietnam.
Email: \texttt{tranthedung56@gmail.com}}%
}

\date{}
\maketitle
\begin{abstract}
We extend Raimi's classical partition theorem to the continuous setting of the circle and $n$-dimensional torus. Building on recent work of Hegyvári, Pach, and Pham in finite groups, we prove that there exist measurable partitions of the $n$-dimensional torus $\mathbb{T}^n$ with the property that for any finite measurable cover, some translated part of the cover has positive measure intersection with every partition element. Our proof adapts combinatorial arguments from the finite setting using measure-theoretic techniques and slicing arguments in product spaces.
\end{abstract}


\section{Introduction}
A classical theorem of Raimi \cite{Raimi} says that there exists a partition of $\mathbb{N}$, $\mathbb{N}=E_1\cup E_2$, such that for every partition of $\mathbb{N}$ into finitely many parts, $\mathbb{N}=\bigcup\limits_{j=1}^t F_j$ with $t\in\mathbb{N},$ we can always find $j\in\{1,2,\ldots,t\}$ and $k\in \mathbb{N}$ such that both $(F_j+k)\cap E_1$ and $(F_j+k)\cap E_2$ are infinite sets. This result was proved by using a topological method. Hindman \cite[p.180, Theorem 11.15]{HM79} later provided an elementary proof and showed that one can choose $E_1$ to be the set of natural numbers whose last non-zero entry in their ternary expansion is $1$ and $E_2=\mathbb{N}\setminus E_1$. 

Strengthened analogs specifying the densities of the partition sets or guaranteeing positive densities in the conclusion can be found in the work of Hegyv\'{a}ri \cite{NH}, or the work of Bergelson and Weiss \cite{Bergelson2}.

In a recent remarkable paper, Hegyvári, Pach, and Pham \cite{HPP} developed a framework that combines tools from harmonic analysis, additive combinatorics, and group theory to obtain polynomial and finite-group extensions of Raimi's theorem. They also made explicit its link to Ramsey theory: \textit{while classical Ramsey results ask ``what patterns are forced by density or by finite colorings?'', Raimi's perspective asks ``which partitions are unavoidable in any finite coloring after an appropriate shift?''}. One of their main theorems is the following.

\begin{theorem}[Hegyv\'{a}ri--Pach--Pham {\cite{HPP}}]\label{thmpolynomial-main}
Let $r, t, k, f\in \mathbb{N}$. Let $P^{(1)},\dots,P^{(f)}\in\mathbb{Z}[x]$ be non-constant polynomials with the properties that
\[
P^{(1)}(0)=\cdots=P^{(f)}(0)=0
\]
and the leading coefficients are positive.

There exists a partition
\[
\mathbb{N}^k = \bigcup_{i=1}^r E_i
\]
such that for any partition
\[
\mathbb{N}^k = \bigcup_{j=1}^t F_j,
\]
there exist \(m\in\{1,\dots,t\}\), an element $x_0\in \mathbb{N}$, and a set $H\subset\mathbb{N}$ of positive lower density with the property that for every $h\in H$ and every $j\in\{1,\dots,f\}$, the set
\[
\big(F_m+(x_0+P^{(j)}(h))\mathbf{1}_k\big)\cap E_i
\]
is infinite, where $\mathbf{1}_k=(1, \ldots, 1)\in \mathbb{N}^k$.
\end{theorem}
It follows from their result that polynomial patterns are also unavoidable, which is parallel to Bergelson and Leibman’s polynomial Szemer\'{e}di theorem \cite{BL} in Ramsey theory.

In the finite setting, Hegyv\'{a}ri, Pach, and Pham \cite{HPP} proved extensions in both abelian and non-abelian groups. Here, we recall their result for abelian groups which is relevant to our main results and techniques.

\begin{theorem}[Hegyv\'{a}ri--Pach--Pham {\cite{HPP}}]\label{thm3.1}
Let $r, t\in \mathbb{N}$. Let $G$ be a finite cyclic group of order $N$, written additively. There are constants \(\alpha=\alpha(r, t)>0\) and \(N_0=N_0(r,t)\) such that the following holds for every \(N\ge N_0\). There exists a partition
\[
G = \bigcup_{i=1}^r E_i
\]
such that for each partition
\[
G = \bigcup_{j=1}^t F_j,
\]
there exist \(m\in\{1,\dots,t\}\) and an element $h\in G$ such that
\[
\forall i\in\{1,2,\ldots,r\},\quad |(F_m+h)\cap E_i|\ge \alpha |G|.
\]
In general, if $G$ is an abelian group with exponent at least $N_0$, then a conclusion with a weaker constant holds, namely,
\[
\forall i\in\{1,2,\ldots,r\},\quad |(F_m+h)\cap E_i|\ge \frac{\alpha}{t} |G|.
\]
\end{theorem}

A natural question is whether the phenomenon discovered by Hegyvári, Pach, and Pham persists in compact abelian groups. In this paper, we answer this question affirmatively for the circle and, more generally, for the $n$-dimensional torus.

Let $C$ denote the unit circle, identified with the compact abelian group $\mathbb{R}/\mathbb{Z}$, and let $\mu_1$ be the normalized Haar measure on $C$ satisfying $\mu_1(C)=1$. For $\theta \in C$, define the rotation
\[
R_{\theta}(x) := x + \theta \pmod{1}.
\]

Our first result extends Theorem \ref{thm3.1} to the circle $C\cong \mathbb{T}$.

\begin{theorem}\label{thm:circle-raimi1}
Let $r, t\ge 2$ be numbers in $\mathbb{N}$. There exists a measurable partition
\[
C = \bigcup_{i=1}^r E_i
\]
such that for every finite measurable cover
\[
C \subset F_1 \cup \cdots \cup F_t,
\]
there exist an index $m \in \{1,\dots,t\}$ and a rotation $R_{\theta}$ satisfying
\begin{align*}
    \mu_1\big(R_{\theta}(F_m) \cap E_i\big) > 0
\quad\text{for all } 1\le i\le r.
\end{align*}
\end{theorem}

 We next extend this result to higher-dimensional tori. 
Let $C = \mathbb{R}/\mathbb{Z}$ denote the circle with its Haar probability
measure $\mu_1$, and let $\mathbb{T}^{n-1}$ be the $(n-1)$–dimensional torus
with Haar probability measure $\mu_{n-1}$. 
We identify
\[
\mathbb{T}^n \cong C \times \mathbb{T}^{n-1}
\]
and write $\mu_n = \mu_1 \times \mu_{n-1}$ for the Haar probability measure
on $\mathbb{T}^n$. 
For $\theta \in \mathbb{T}^n$, we denote by $R_\theta$ the translation
$x \mapsto x + \theta$ on $\mathbb{T}^n$.


\begin{theorem}\label{thm:general-case}
  Let $r$, $t\ge 2$ be numbers in $\mathbb{N}$. Then there exists a measurable partition
\[
\mathbb{T}^n = \bigcup_{i=1}^r E_i
\]
such that for every finite measurable cover $$
\mathbb{T}^n \subset F_1 \cup \cdots \cup F_t,$$ there exist an index $m \in \{1,\dots,t\}$ and a translation $R_{\theta}$ satisfying
\begin{align*}
    \mu_n\big(R_{\theta}(F_m) \cap E_i\big) > 0
\quad\text{for all } 1\le i\le r.
\end{align*}
\end{theorem}

Extending Raimi-type results to the continuous setting serves several purposes. Firstly, the tori $\mathbb{T}^n$ arise naturally as limit objects for finite cyclic groups and for subsets of the integers, so Theorems~\ref{thm:circle-raimi1} and~\ref{thm:general-case} provide genuine continuous counterparts to Theorem~\ref{thm3.1}. Secondly, from the viewpoint of ergodic theory, our results reveal rigidity properties of measure-preserving transformations: they exhibit unavoidable partition structures that persist under rotations, which are among the most fundamental examples in dynamics. Finally, the continuous setting brings analytic features that are absent in the finite case, such as continuity properties of the map $\theta \mapsto \mu_n(R_\theta(F)\cap E)$ and slicing arguments on product tori. These tools are essential in our proofs and suggest that further extensions (for instance to polynomial or non-abelian actions on compact groups) may be more naturally studied in the compact, continuous framework.

We briefly comment on the relationship between Theorems \ref{thm:circle-raimi1} and \ref{thm:general-case}. When $n=1$, the $n$-dimensional torus $\mathbb{T}^n$ reduces to the circle $\mathbb{T} \cong C$, and in this case a translation in $\mathbb{T}$ coincides with a rotation on $C$. Thus, Theorem \ref{thm:circle-raimi1} may be viewed as the one-dimensional case of Theorem \ref{thm:general-case}. However, we present both results separately to emphasize the different proof strategies: the direct construction for the circle serves as a foundation for the product-space argument needed in higher dimensions.

The proof of Theorem \ref{thm:circle-raimi1} adapts the elegant combinatorial argument from Hegyv\'{a}ri, Pach, and Pham's proof of Theorem \ref{thm3.1} to the continuous setting of the circle. This adaptation requires establishing some basic properties of rotation–intersection measures (Lemma \ref{lem:single-set}), which follow from standard measure-theoretic arguments. The main technical content involves carefully controlling the measure loss through successive rotations. The proof of Theorem \ref{thm:general-case} builds on Theorem \ref{thm:circle-raimi1} by employing a slicing argument in product spaces, mirroring Hegyv\'{a}ri, Pach, and Pham's approach for general abelian groups in the finite setting.

The paper is organized as follows. In Section \ref{section2}, we prove Theorem \ref{thm:circle-raimi1} for the circle by constructing an explicit partition and adapting the discrete combinatorial argument. Section \ref{section3} extends this to the $n$-dimensional torus using a slicing technique. Section \ref{section4} contains concluding remarks and open questions.


\section{Proof of Theorem \texorpdfstring{\ref{thm:circle-raimi1}}{1.4}}\label{section2}
The proof proceeds in three main steps: Measure-theoretic preliminaries (Lemma~\ref{lem:single-set}), explicit construction of the partition, and iterative rotation argument (Claim~\ref{cl:Es-lower}).

We begin by establishing the basic analytic properties of rotation--intersection measures that will be used throughout the proof. The following lemma shows that the function measuring how much a rotated set intersects a fixed set behaves well: it is continuous and satisfies a natural averaging property.

\begin{lemma}\label{lem:single-set}
Let $E,F \subset C$ be measurable sets with $0 < \mu_1(E) < 1$ and $0 < \mu_1(F) \le 1$. Define
\[
f(\theta) := \mu_1\big(R_{\theta}(F) \cap E\big),\qquad \theta\in C.
\]
Then:
\begin{itemize}
    \item[(i)] The function $f$ is continuous on $C$.
    \item[(ii)] One has
    \[
    \int_C f(\theta)\,d\mu_1(\theta) = \mu_1(E)\mu_1(F).
    \]
    In particular, there exists $\theta_0\in C$ such that $f(\theta_0) = \mu_1(E)\mu_1(F)$.
    \item[(iii)] There exists $\theta\in C$ such that $0 < f(\theta) < \mu_1(F)$.
\end{itemize}
\end{lemma}

\begin{proof}
Define
\[
f(\theta) = \mu_1(R_{\theta}(F)\cap E)
= \int_C \mathbf{1}_E(x)\,\mathbf{1}_F(x - \theta)\,d\mu_1(x),
\]
where $\mathbf{1}_E$ and $\mathbf{1}_F$ denote the indicator functions of $E$ and $F$, respectively.

\medskip
{\it Proof of (i).} Fix $\varepsilon > 0$. Since continuous functions are dense in $L^1(C)$, there exists a continuous function $h: C \to \mathbb{R}$ such that
\begin{equation}\label{eq:approx-1F-h}
\|\mathbf{1}_F - h\|_{L^1(C)} < \varepsilon.
\end{equation}
Define the approximating function
\[
f_h(\theta) := \int_C \mathbf{1}_E(x)\,h(x - \theta)\,d\mu_1(x).
\]
For each fixed $x$, the map $\theta \mapsto h(x-\theta)$ is continuous, and the integrand is bounded by $\|\mathbf{1}_E\|_\infty \|h\|_\infty$. Thus, by dominated convergence, $f_h$ is continuous on $C$.

For each $\theta\in C$, using the change of variables $y=x-\theta$ and the rotation invariance of $\mu_1$, we have
\[
\begin{aligned}
|f(\theta) - f_h(\theta)|
&\leq \int_C \mathbf{1}_E(x)\,\big|\mathbf{1}_F(x-\theta) - h(x-\theta)\big|\,d\mu_1(x) \\
&\leq \int_C \big|\mathbf{1}_F(y) - h(y)\big|\,d\mu_1(y)
= \|\mathbf{1}_F - h\|_{L^1(C)} < \varepsilon.
\end{aligned}
\]
Thus, $\sup_{\theta\in C}|f(\theta)-f_h(\theta)|<\varepsilon$, so $f_h\to f$ uniformly. Since each $f_h$ is continuous, it follows that $f$ is continuous as well. This proves {\it (i)}.

{\it Proof of (ii).} By Fubini's theorem and rotation invariance of $\mu_1$,
\begin{align}\label{identity-f-mu-E-E}
\nonumber
\int_C f(\theta)\,d\mu_1(\theta)
=& \int_C \left(\int_C \mathbf{1}_E(x)\,\mathbf{1}_F(x - \theta)\,d\mu_1(x)\right)d\mu_1(\theta) \\ \nonumber
= &\int_C \mathbf{1}_E(x)\left(\int_C \mathbf{1}_F(x-\theta)\,d\mu_1(\theta)\right)d\mu_1(x) \\
=& \mu_1(E)\mu_1(F),
\end{align}
which proves {\it (ii)}.

{\it Proof of (iii).} We now prove {\it (iii)} by contradiction. Suppose that
\[
f(\theta) \in \{0,\mu_1(F)\} \quad \text{for all }\theta\in C.
\]
Since $f$ is continuous and $C$ is connected, the image set $f(C)$ must be a connected subset of $\{0,\mu_1(F)\}$, hence
\[
f(C) \subset \{0\} \quad \text{or} \quad f(C)\subset\{\mu_1(F)\}.
\]
We have 
\[
\int_{C} f(\theta) \, d\mu_1(\theta) = 
\begin{cases}
0 & \text{if } f \equiv 0, \\
\mu_1(F) & \text{if } f \equiv \mu_1(F),
\end{cases}
\]
which contradicts the identity \eqref{identity-f-mu-E-E} since $0<\mu_1(E)<1$. Thus, $f$ cannot take only the values $0$ and $\mu_1(F)$. Consequently, there exists $\theta\in C$ such that $ 0 < f(\theta) < \mu_1(F)$. This proves \it{(iii)}.
\end{proof}

\begin{remark}
For the applications below, we will use both the existence of $\theta$ with
\[
0 < \mu_1(R_{\theta}(F)\cap E) < \mu_1(F)
\]
and the existence of $\theta$ with
\[
\mu_1(R_{\theta}(F)\cap E) = \mu_1(E)\mu_1(F).
\]
Both statements follow from Lemma \ref{lem:single-set}.
\end{remark}

We now define the partition of the circle that will satisfy the conclusion of Theorem~\ref{thm:circle-raimi1}. The construction is motivated by the need to balance two competing requirements: on one hand, we want the partition elements to have substantial measure so that intersections can be detected; on the other hand, we need the intervals to be arranged so that small rotations can shift mass from one partition element to the next without completely evacuating the previous elements.

The solution is to use a \emph{geometrically decreasing} partition: each successive interval $E_i$ has measure smaller than the previous one by a factor of $k$, where $k$ is chosen large enough to absorb the measure losses incurred during the iterative rotation process.


Throughout this section, we identify $C$ with the interval $[0,1)$ equipped with Lebesgue measure, which we still denote by $\mu_1$.

Fix integers $r$, $t\ge 2$. Set
\[
k \ge 1+2^{r+4}t,
\qquad
S_k := 1 + \frac{1}{k} + \dots + \frac{1}{k^{r-1}}.
\]
Define
\[
\Delta_1 := \frac{1}{S_k}, \qquad
\Delta_i := \frac{\Delta_{i-1}}{k} \quad\text{for } 2\le i\le r.
\]
Then
\[
\sum_{i=1}^r \Delta_i
= \Delta_1\left(1 + \frac{1}{k} + \cdots + \frac{1}{k^{r-1}}\right)
= \frac{1}{S_k}\cdot S_k = 1.
\]
We define intervals
\[
E_i := [u_i,u_{i+1}), \qquad u_i := \sum_{j=1}^{i-1}\Delta_j,
\]
for $2\le i\le r$. Then the intervals $E_1,\dots,E_r$ are pairwise disjoint and form a partition of $[0,1)$, i.e, 
\[
[0,1) = \bigsqcup_{i=1}^r E_i.
\]

\begin{proof}[Proof of Theorem~\ref{thm:circle-raimi1}]
Throughout the proof we identify $C$ with the interval $[0,1)$ equipped with Lebesgue measure, still denoted by $\mu_1$.

\medskip\noindent

Suppose we are given a finite measurable cover
\[
C \subset F_1 \cup \cdots \cup F_t.
\]
By the assumptions on $C$, monotonicity, and additivity of $\mu_1$, we have
\[
1 = \mu_1(C) \le \sum_{j=1}^t \mu_1(F_j),
\]
so there exists $m\in\{1,\dots,t\}$ with
\[
\mu_1(F_m) \ge \frac{1}{t} =: \beta.
\]

Apply {\it (ii)} of Lemma~\ref{lem:single-set} to the sets $E_1$ and $F_m$. Then there exists $\theta_1\in C$ such that
\begin{equation}\label{eq:phi1-E1}
\mu_1\big(R_{\theta_1}(F_m)\cap E_1\big)
= \mu_1(E_1)\mu_1(F_m)
\ge \beta\,\mu_1(E_1)
= \beta\,\Delta_1.
\end{equation}
We now begin a recursive construction of rotations
\[
\theta_1,\theta_2,\dots,\theta_r\in C,
\qquad
\phi_s := \theta_1+\cdots+\theta_s\quad(1\le s\le r),
\]
with the following properties.

\begin{claim}\label{cl:Es-lower}
For each $1\le s\le r,$ we can choose $\theta_1,\dots,\theta_s$ so that
\begin{equation}\label{eq:Es-lower-phi-s}
\mu_1\big(R_{\phi_s}(F_m)\cap E_s\big)
\;\ge\; \frac{\beta}{2^{s+2}}\,\mu_1(E_s)
= \frac{\beta}{2^{s+2}}\,\Delta_s,
\end{equation}
and, for $1\le s\le r-1$,
\begin{equation}\label{eq:Es-lower-phi-s+1}
\mu_1\big(R_{\phi_{s+1}}(F_m)\cap E_s\big)
\;\ge\; \frac{\beta}{2^{s+4}}\,\mu_1(E_s)
= \frac{\beta}{2^{s+4}}\,\Delta_s.
\end{equation}
\end{claim}

We first verify the claim for $s=1$, and then show how to pass from $s$ to $s+1$.

\medskip\noindent
{\bf Base case ($s=1$).}  It follows from \eqref{eq:phi1-E1} with $\phi_1 := \theta_1$ that \eqref{eq:Es-lower-phi-s} holds for $s=1$.

Next we construct $\theta_2$ and obtain the lower bound \eqref{eq:Es-lower-phi-s+1} for $s=1$. We decompose $E_1$ into $k$ pairwise disjoint subintervals of equal length, each satisfying $|I_i|=\frac{\Delta_1}{k}=\Delta_2,$ i.e.,
\[
E_1 = \bigsqcup_{i=1}^k I_i,
\qquad
I_i := \Big[\frac{(i-1)\Delta_1}{k},\,\frac{i\Delta_1}{k}\Big),
\quad 1\le i\le k.
\] 
By the pigeonhole principle and \eqref{eq:phi1-E1}, there exists $i$ such that
\[
\mu_1\big(R_{\theta_1}(F_m) \cap I_i\big)
\ge \frac{1}{k}\,\mu_1\big(R_{\theta_1}(F_m)\cap E_1\big)
\ge \frac{\beta}{k}\,\Delta_1.
\]
In particular, since $\frac{\beta}{k}\Delta_1\ge \frac{\beta}{2}|I_i|$, we may choose $j\in\{1,\dots,k\}$ such that
\begin{equation}\label{eq:Ij-large-1}
\mu_1\big(R_{\theta_1}(F_m) \cap I_j\big)
\ge \frac{\beta}{2}\,|I_j|
= \frac{\beta}{2}\,\Delta_2.
\end{equation}
Among all such indices $j$, pick the \emph{rightmost} one, i.e., the maximal $j$ satisfying \eqref{eq:Ij-large-1}. Then for every $i>j$ we have
\[
\mu_1\big(R_{\theta_1}(F_m) \cap I_i\big)
< \frac{\beta}{2}\,|I_i|
= \frac{\beta}{2}\,\Delta_2,
\]
and hence
\begin{equation}\label{eq:tail-mass-E1}
\sum_{i=j+1}^k \mu_1\big(R_{\theta_1}(F_m) \cap I_i\big)
\le \frac{\beta}{2}\sum_{i=j+1}^k |I_i|
\le \frac{\beta}{2}\,\Delta_1.
\end{equation}

Let $I_j=[a_j,b_j)$ and set
\[
\theta_2 := u_2 - a_j = \Delta_1 - a_j.
\]
Then $0<\theta_2\le \Delta_1$ and $I_j+\theta_2 = E_2$. Moreover, for $i<j$ the intervals $I_i+\theta_2$ are still contained in $E_1=[0,\Delta_1)$: their right endpoints are at most
\[
\frac{(j-1)\Delta_1}{k} + \theta_2
= \frac{(j-1)\Delta_1}{k} + \Delta_1 - a_j
\le \Delta_1.
\]

Define $\phi_2 := \theta_1 + \theta_2$. Using \eqref{eq:Ij-large-1}, we obtain
\[
\mu_1\big(R_{\phi_2}(F_m)\cap E_2\big)
\ge \mu_1\big(R_{\theta_1}(F_m)\cap I_j\big)
\ge \frac{\beta}{2}\,\Delta_2
\ge \frac{\beta}{2^{3}}\,\Delta_2,
\]
so \eqref{eq:Es-lower-phi-s} holds for $s=2$.

Next we estimate the intersection with $E_1$ after the second rotation. Any point of $R_{\theta_1}(F_m)\cap E_1$ that leaves $E_1$ under the shift by $\theta_2$ must belong to $I_j$ or to one of the intervals $I_{j+1},\dots,I_k$. Therefore,
\[
\mu_1\big(R_{\phi_2}(F_m)\cap E_1\big)
\ge \mu_1\big(R_{\theta_1}(F_m)\cap E_1\big)
- \mu_1\big(R_{\theta_1}(F_m)\cap I_j\big)
- \sum_{i=j+1}^k \mu_1\big(R_{\theta_1}(F_m)\cap I_i\big).
\]
Combining this with \eqref{eq:phi1-E1} and \eqref{eq:tail-mass-E1}, we get
\[
\mu_1\big(R_{\phi_2}(F_m)\cap E_1\big)
\ge \beta\Delta_1 -\Delta_2 - \frac{\beta}{2}\,\Delta_1\ge \frac{\beta\Delta_1}{4},
\]
where we have used $\Delta_2=\frac{\Delta_1}{k}$ and $k>1+2^{r+4} \cdot \frac{1}{\beta}$ in the last inequality. In particular,
\[
\mu_1\big(R_{\phi_2}(F_m)\cap E_1\big)
\ge \frac{\beta}{2^{4}}\,\Delta_1,
\]
so \eqref{eq:Es-lower-phi-s+1} holds for $s=1$. 
This verifies Claim~\ref{cl:Es-lower} for $s=1$.

\medskip\noindent
{\bf General case: $s\to s+1$.}
Assume $1\le s\le r-1$ and that rotations $\theta_1,\dots,\theta_s$ have already been chosen so that \eqref{eq:Es-lower-phi-s} holds. Set $\phi_s:=\theta_1+\cdots+\theta_s$. We now construct $\theta_{s+1}$ and $\phi_{s+1}:=\phi_s+\theta_{s+1}$ and show that \eqref{eq:Es-lower-phi-s} holds for $s+1$, and that \eqref{eq:Es-lower-phi-s+1} holds for this $s$.

Partition $E_s$ into $k$ equal pairwise disjoint subintervals $J_{i,s}$ with length $|J_{i,s}| = \frac{\Delta_s}{k} = \Delta_{s+1}$:
\[
E_s = \bigsqcup_{i=1}^k J_{i,s},
\qquad
J_{i,s}
:= \Big[u_s + \frac{(i-1)\Delta_s}{k},\, u_s + \frac{i\Delta_s}{k}\Big).
\quad 1\le i\le k,
\]
By the pigeonhole principle and \eqref{eq:Es-lower-phi-s}, there exists $i$ such that
\[
\mu_1\big( R_{\phi_s}(F_m) \cap J_{i,s}\big)
\ge \frac{1}{k}\,\mu_1\big(R_{\phi_s}(F_m)\cap E_s\big)
\ge \frac{1}{k}\cdot\frac{\beta}{2^{s+2}}\,\Delta_s
= \frac{\beta}{2^{s+2}}\,\Delta_{s+1}.
\]
In particular, since $\frac{\beta}{2^{s+2}}\ge \frac{\beta}{2^{s+3}}$, we may choose $j\in\{1,\dots,k\}$ such that
\begin{equation}\label{eq:Jjs-large}
\mu_1\big(R_{\phi_s}(F_m) \cap J_{j,s}\big)
\ge \frac{\beta}{2^{s+3}}\,\Delta_{s+1}.
\end{equation}
Among all such indices $j$, pick the \emph{rightmost} one, i.e., the maximal $j$ satisfying \eqref{eq:Jjs-large}. Then for every $i>j$ we have
\[
\mu_1\big(J_{i,s}\cap R_{\phi_s}(F_m)\big)
< \frac{\beta}{2^{s+3}}\,|J_{i,s}|
= \frac{\beta}{2^{s+3}}\,\Delta_{s+1},
\]
and hence
\begin{equation}\label{eq:tail-mass-Es}
\sum_{i=j+1}^k \mu_1\big(J_{i,s}\cap R_{\phi_s}(F_m)\big)
\le \frac{\beta}{2^{s+3}}\sum_{i=j+1}^k |J_{i,s}|
\le \frac{\beta}{2^{s+3}}\,\Delta_s.
\end{equation}

Let $J_{j,s}=[c_{j,s},d_{j,s})$ and define $\theta_{s+1} := u_{s+1}-c_{j,s}$.
Then $0<\theta_{s+1}\le \Delta_s$ and
\[
J_{j,s}+\theta_{s+1} = E_{s+1}.
\]

A simple geometric observation shows that, for all $i < j$, the intervals $J_{i,s} + \theta_{s+1}$ remain inside $E_s$: their right endpoints do not exceed $u_{s+1}$, and their left endpoints are at least $u_s$. From the definition of $\theta_{s+1}$ and \eqref{eq:Jjs-large}, we get
\[
\mu_1\big(R_{\phi_{s+1}}(F_m)\cap E_{s+1}\big)
\ge \mu_1\big(R_{\phi_s}(F_m)\cap J_{j,s}\big)
\ge \frac{\beta}{2^{s+3}}\,\Delta_{s+1}
= \frac{\beta}{2^{(s+1)+2}}\,\Delta_{s+1}.
\]
This is precisely \eqref{eq:Es-lower-phi-s} with $s+1$ in place of $s$.

\medskip\noindent
We prove \eqref{eq:Es-lower-phi-s+1} as below. In fact, any point of $R_{\phi_s}(F_m)\cap E_s$ that leaves $E_s$ under the shift by $\theta_{s+1}$ must belong to $J_{j,s}$ or to one of the intervals $J_{j+1,s},\dots,J_{k,s}$. Therefore,
\[
\mu_1\big(R_{\phi_{s+1}}(F_m)\cap E_s\big)
\ge \mu_1\big(R_{\phi_s}(F_m)\cap E_s\big)
- \mu_1\big(R_{\phi_s}(F_m)\cap J_{j,s}\big)
- \sum_{i=j+1}^k \mu_1\big(R_{\phi_s}(F_m)\cap J_{i,s}\big).
\]
Using \eqref{eq:Es-lower-phi-s}, the trivial bound $\mu_1(R_{\phi_s}(F_m)\cap J_{j,s})\le |J_{j,s}|=\Delta_{s+1}=\frac{\Delta_s}{k}$, \eqref{eq:tail-mass-Es}, and $\frac{1}{k} \leq \frac{\beta}{2^{s+5}}$ which follows from $k>1+2^{r+4}t$, we obtain
\[
\mu_1\big(R_{\phi_{s+1}}(F_m)\cap E_s\big)
\ge \frac{\beta}{2^{s+2}}\,\Delta_s - \frac{\Delta_s}{k} - \frac{\beta}{2^{s+3}}\,\Delta_s
\ge \frac{\beta}{2^{s+4}}\,\Delta_s,
\]
which is \eqref{eq:Es-lower-phi-s+1}. This completes the proof of Claim~\ref{cl:Es-lower} by induction on $s$.

\medskip\noindent
To complete the proof of Theorem \ref{thm:circle-raimi1}, it remains to show that
\[
\mu_1\big(R_{\phi_r}(F_m)\cap E_i\big) > 0
\qquad\text{for all }1\le i\le r.
\]
Applying Claim~\ref{cl:Es-lower} with $s=r$, we obtain a rotation $\phi_r = \theta_1+\cdots+\theta_r$ such that
\[
\mu_1\big(R_{\phi_r}(F_m)\cap E_r\big)
\ge \frac{\beta}{2^{r+2}}\,\mu_1(E_r) > 0.
\]
Thus, it suffices to show that for each $1\le s\le r-1$,
\begin{equation}\label{finalequ}
\mu_1\big(R_{\phi_r}(F_m)\cap E_s\big)
\ge \frac{\beta}{2^{s+5}}\,\mu_1(E_s) > 0.\end{equation}
To verify the estimate \eqref{finalequ}, we first note that
\begin{align}\label{HKeq}
\mu_1\big(R_{\phi_r}(F_m)\cap E_s\big)
&\ge \mu_1\big(R_{\phi_{s+1}}(F_m)\cap E_s\big)-\sum_{i=s+2}^r\theta_i \nonumber\\ 
&\ge \frac{\beta}{2^{s+4}}\,\mu_1(E_s) -\sum_{i=s+2}^r\theta_i,
\end{align}
where  the last inequality follows by using  the second part \eqref{eq:Es-lower-phi-s+1} of Claim~\ref{cl:Es-lower}.  

From the fact that $k \ge 1+2^{r+4}t \ge 1+\frac{2^{s+5}}{\beta}$, we obtain the estimate
\begin{align*}
\sum_{i=s+2}^r\theta_i &\le \Delta_{s+1}+\cdots+\Delta_{r-1} \le \Delta_{s+1}\left(1+\frac{1}{k}+\cdots+\frac{1}{k^{r-s}}\right)\\
& \le \frac{k\Delta_{s+1}}{k-1}\le \frac{\Delta_s}{k-1}=\frac{\mu_1(E_s)}{k-1}\le \frac{\beta}{2^{s+5}}\mu_1(E_s).
\end{align*}
Combining this estimate with \eqref{HKeq}, we obtain that
\[
\mu_1\big(R_{\phi_r}(F_m)\cap E_s\big) > \frac{\beta}{2^{s+4}}\,\mu_1(E_s)- \frac{\beta}{2^{s+5}}\,\mu_1(E_s)=\frac{\beta}{2^{s+5}} \,\mu_1(E_s)>0.
\]

In conclusion, we get
\[
\mu_1\big(R_{\theta}(F_m)\cap E_i\big) > 0
\qquad\text{for all }1\le i\le r,
\]
where $\theta=\phi_r$.

Thus, for the fixed partition $\{E_i\}_{i=1}^r$ and for any finite cover $C\subset F_1\cup\cdots\cup F_t$, we have found $m\in\{1,\dots,t\}$ and a rotation $\theta=\phi_r\in C$ such that
\[
\mu_1\big(R_{\theta}(F_m)\cap E_i\big) > 0
\qquad\text{for all }1\le i\le r.
\]
This completes the proof of Theorem~\ref{thm:circle-raimi1}.
\end{proof}

\section{Proof of Theorem \texorpdfstring{\ref{thm:general-case}}{1.5}}\label{section3}

In this section, we provide the proof of Theorem \ref{thm:general-case}, which extends the one-dimensional result (Theorem~\ref{thm:circle-raimi1}) to the $n$-dimensional torus $\mathbb{T}^n$. The proof employs a \emph{slicing argument} that reduces the higher-dimensional problem to the circle case, adapting the approach of Hegyv\'ari, Pach, and Pham~\cite{HPP} for general abelian groups in the finite setting.

Now we restate and prove Theorem \ref{thm:general-case} as follows.
\begin{theorem}\label{theorem-restate-general}
For all integers $r, t \geq 2$, there exists a measurable partition $\{E_i\}_{i=1}^{r}$ of the $n$-dimensional torus $\mathbb{T}^n$ with the following property: given any finite measurable cover $\{F_j\}_{j=1}^{t}$ of $\mathbb{T}^n$, one can find an index $m \in \{1, \ldots, t\}$ and a translation $R_{\theta}$ such that
\[
\mu_n\bigl(R_{\theta}(F_m)\cap E_i\bigr) > 0 \quad \text{for every } i \in \{1, \ldots, r\}.
\]
\end{theorem}

\begin{proof}
We write $\mathbb{T}^n = C \times \mathbb{T}^{\,n-1}$ and denote the zero element of $\mathbb{T}^{\,n-1}$ by $0_{\mathbb{T}^{n-1}}$. Fix $r,t \in \mathbb{N}$ and let $E_1,\dots,E_r$ be the measurable partition of $C$ constructed in Theorem~\ref{thm:circle-raimi1}. For each $i$, define
\[
    \bar{E}_i := E_i \times \mathbb{T}^{\,n-1} \subset \mathbb{T}^n.
\]
Then $\bigcup\limits_{i=1}^r \bar{E}_i$ is a measurable partition of $\mathbb{T}^n$. Given any measurable cover $\mathbb{T}^n \subset  \bigcup\limits_{m=1}^t F_m$, we will show that there exist $m^\ast \in \{1,\dots,t\}$ and $\theta \in \mathbb{T}^n$ such that
\[
    \mu_n\big( R_\theta(F_{m^\ast}) \cap \bar{E}_i \big) > 0
    \qquad \text{for all } i = 1,\dots,r.
\]
To establish this, fix $x \in C$ and $m \in \{1,\dots,t\}$ and set
\[
    A_m(x) := \{ y \in \mathbb{T}^{\,n-1} : (x,y) \in F_m \}.
\]
By Fubini’s theorem, for each fixed $m$, the set $A_m(x)$ is measurable for almost every $x \in C$, and the function $x \mapsto \mu_{n-1}(A_m(x))$ is measurable. Since $\mathbb{T}^n \subset \bigcup\limits_{m=1}^t F_m$, the sets $\{A_m(x)\}_{m=1}^t$ form a cover of $\mathbb{T}^{\,n-1}$ for each fixed $x$, and we have
\[
    \sum_{m=1}^t \mu_{n-1}(A_m(x)) \geq \mu_{n-1}(\mathbb{T}^{\,n-1})=1
    \qquad \text{for all } x \in C.
\]
By the pigeonhole principle, for each $x$ there exists at least one index $m_0$ such that
\[
    \mu_{n-1}(A_{m_0}(x)) \ge \frac{1}{t}.
\]
By choosing the minimal such index, select one such $m$ for each $x$ and denote it by $m(x)$. Define
\[
    C_m := \{ x \in C : m(x) = m \},
    \qquad m = 1,\dots,t.
\]
Then $\{C_m\}_{m=1}^t$ is a measurable partition of $C$.

Applying Theorem~\ref{thm:circle-raimi1} to the partition $\{C_m\}_{m=1}^t$ of $C$, we obtain an index $m^\ast \in \{1,\dots,t\}$ and a rotation $\theta_1 \in C$ such that
\begin{equation}\label{step1}
    \mu_{1}\big( R_{\theta_1}(C_{m^\ast}) \cap E_i \big) > 0
    \qquad \text{for all } i=1,\dots,r.
\end{equation}
We now use this to establish the corresponding property for $\mathbb{T}^n$.

Now set $\theta := (\theta_1, 0_{\mathbb{T}^{n-1}}) \in \mathbb{T}^n$. For any fixed $i$, by using the product structure of Haar measure on $\mathbb{T}^n = C \times \mathbb{T}^{\,n-1}$ and Fubini's theorem, the measure of $R_\theta(F_{m^\ast}) \cap \bar{E}_i$ decomposes as
\begin{align*}
    \mu_{n}\big( R_\theta(F_{m^\ast}) \cap \bar{E}_i \big)
    =& \int_{x \in E_i}
      \mu_{n-1}\big( \{ y \in \mathbb{T}^{\,n-1} :
    (x-\theta_1, y) \in F_{m^\ast} \} \big)\, d\mu_1(x) \\
    =& \int_{x \in E_i} \mu_{n-1}(A_{m^{\ast}}(x-\theta_1)) \, d\mu_1(x),
\end{align*}
where we have used $\{ y \in \mathbb{T}^{\,n-1} : (x-\theta_1, y) \in F_{m^\ast} \} =A_{m^{\ast}}(x-\theta_1)$ in the last equality.

Restricting the integral to those $x \in E_i$ such that
$x - \theta_1 \in C_{m^\ast}$ yields
\[
    \mu_{n}\big( R_{\theta}(F_{m^\ast}) \cap \bar{E}_i \big)
    \ge
    \int\limits_{\substack{x \in E_i \\ x - \theta_1 \in C_{m^\ast}}}
        \mu_{n-1}(A_{m^\ast}(x - \theta_1)) \, d\mu_1(x).
\]
But if $x - \theta_1 \in C_{m^\ast}$, by definition of $C_{m^\ast}$,
\[
    \mu_{n-1}(A_{m^\ast}(x - \theta_1))
    \ge \frac{1}{t}.
\]
Hence,
\[
    \mu_{n}\big( R_{\theta}(F_{m^\ast}) \cap \bar{E}_i \big)
    \ge
    \frac{1}{t}
    \cdot
    \mu_1\big( R_{\theta_1}(C_{m^\ast}) \cap E_i \big)
    > 0,
\]
where the positivity follows from \eqref{step1}. This holds for each $1 \leq i \leq r$, completing the proof of Theorem \ref{theorem-restate-general}.
\end{proof}

\section{Some remarks and open questions}\label{section4}

Our results raise several further questions. 

Firstly, Theorems~\ref{thm:circle-raimi1} and~\ref{thm:general-case} concern linear translations $x\mapsto x+\theta$ and do not incorporate polynomial parameters. It would be very interesting to obtain compact analogs of the polynomial extensions in \cite{HPP}. For example, given non-constant polynomials $P^{(1)},\dots,P^{(f)}\in\mathbb{Z}[x]$ with $P^{(j)}(0)=0$, does there exist a measurable partition $\mathbb{T}^n=\bigcup\limits_{i=1}^r E_i$ with the following property: for every finite measurable cover $\mathbb{T}^n \subset F_1 \cup \cdots \cup F_t$ there exist an index $m\in\{1,\dots,t\}$ and a set $H\subset\mathbb{N}$ of positive lower density such that, for every $h\in H$ and every $j\in\{1,\dots,f\}$,
\[
\mu_n\big(R_{P^{(j)}(h)\mathbf{1}_n}(F_m)\cap E_i\big) > 0
\quad\text{for all } 1\le i\le r?
\]

Secondly, our arguments are qualitative and do not provide sharp quantitative information. It is natural to ask whether one can obtain uniform lower bounds
\[
\mu_n\big(R_\theta(F_m)\cap E_i\big)\ge c(r,t,n)>0
\]
that are independent of the particular cover, or even determine the optimal dependence of such constants on $r$, $t$, and $n$.

Thirdly, do Raimi-type unavoidable partition phenomena persist for more general compact groups, such as compact Lie groups or profinite groups, under translation actions or other natural classes of measure-preserving transformations?

Another natural direction would be to investigate whether Raimi-type partition phenomena extend to the unit sphere $\mathbb{S}^{n-1} \subset \mathbb{R}^n$. Unlike the torus, the sphere does not possess a natural group structure, so one must first decide what class of transformations to consider. A natural question is: does there exist a measurable partition $\mathbb{S}^{n-1} = \bigcup\limits_{i=1}^r E_i$ such that for every finite measurable cover $\mathbb{S}^{n-1} \subset F_1 \cup \cdots \cup F_t$, there exist an index $m \in \{1, \ldots, t\}$ and a rotation $g \in \mathrm{SO}(n)$ satisfying $\mu(g(F_m) \cap E_i) > 0$ for all $i \in \{1, \ldots, r\}$, where $\mu$ denotes the normalized surface measure on $\mathbb{S}^{n-1}$? This question would connect Raimi-type results to the representation theory of compact Lie groups and spherical harmonic analysis, potentially revealing new rigidity phenomena in geometric measure theory.

Finally, it has been asked in~\cite{HPP} whether Theorem~\ref{thm3.1} can be applied to the Erd\H{o}s distance problem over finite fields. If $G = \mathbb{F}_q^d$, then Theorem~\ref{thm3.1} implies that one can find a partition
\[
\mathbb{F}_q^d = \bigcup_{i=1}^r E_i
\]
such that for each partition $\mathbb{F}_q^d = \bigcup\limits_{j=1}^t F_j$, there exist $m \in \{1, \ldots, t\}$ and an element $h \in \mathbb{F}_q^d$ such that
\[
\forall i \in \{1, 2, \ldots, r\}, \quad |(F_m + h) \cap E_i| \geq \alpha q^d.
\]
It is an interesting open problem to determine whether this structural result can be used to obtain better understanding of the distance set in $\mathbb{F}_q^d$. Recall that for $E \subset \mathbb{F}_q^d$, the distance between two vectors $x$ and $y$ is defined by $\|x - y\| = (x_1 - y_1)^2 + \cdots + (x_d - y_d)^2$, and the distance set is
\[
\Delta(E) := \{\|x - y\| : x, y \in E\}.
\]
The Erd\H{o}s-Falconer distance conjecture in the finite field setting~\cite{IR} asserts that if $d \geq 2$ is even and $|E| \gg q^{d/2}$, then the distance set $\Delta(E)$ covers a positive proportion of all elements in the field. This is the discrete analog of the celebrated Falconer distance conjecture in the continuous setting, which states that for any compact set $E\subset \mathbb{R}^d$ of Hausdorff dimension greater than $\frac{d}{2}$, the distance set $\Delta(E)$ is of positive Lebesgue measure. In the planar case $d=2$, the best currently known threshold in both settings is $5/4$, established in two influential works: Guth, Iosevich, Ou, and Wang~\cite{guth} in the Euclidean setting, and Murphy, Petridis, Pham, Rudnev, and Stevens~\cite{MPPRS} in the finite field setting.

In this direction, it is worth mentioning the related work of Pham and Yoo~\cite{PY}, in which they studied two-way connections between intersection-type problems and distance questions. In particular, they considered the following complementary questions:

\begin{question}
Let $A$ and $B$ be subsets of $\mathbb{F}_q^d$ and let $g \in O(d)$. Under what conditions on $A$, $B$, and $g$ does one have
\[
|A \cap (g(B) + z)| \geq \frac{|A||B|}{q^d}
\]
for almost every $z \in \mathbb{F}_q^d$, or even the stronger conclusion
\[
|A \cap (g(B) + z)| \sim \frac{|A||B|}{q^d}
\]
for almost every $z \in \mathbb{F}_q^d$?
\end{question}

\begin{question}
Let $P$ be a subset of $\mathbb{F}_q^{2d}$ and let $g \in O(d)$. Define
\[
S_g(P) := \{x - gy : (x, y) \in P\} \subset \mathbb{F}_q^d.
\]
Under what conditions on $P$ and $g$ does one have $|S_g(P)| \gg q^d$?
\end{question}

In the plane over prime fields, using a deep $L^2$ distance bound established in \cite{MPPRS}, they obtained strong results on the above two questions. As consequences, they most notably proved a rotational Erd\H{o}s-Falconer distance problem and resolved the prime field case of a question raised by Mattila~\cite{Mat23} in the continuous setting. We refer the reader to~\cite{PY} for detailed discussions.

It would be interesting to explore whether Theorems~\ref{thm:circle-raimi1} and~\ref{thm:general-case} can be applied to study other analogs of the distance problem, or whether the partition structures guaranteed by Raimi-type theorems can yield new insights into expansion phenomena in finite groups.

{\bf Acknowledgement.} 
D. Koh was supported by the National Research Foundation of Korea (NRF) grant funded by the Korea government (MSIT) (NO. RS-2023-00249597). D. T. Tran would like to thank the Vietnam Institute for Advanced Study in Mathematics (VIASM) for its hospitality and excellent working environment.


\begin{thebibliography}{99}



\bibitem{Bergelson2} V. Bergelson and B. Weiss, \textit{Translation Properties of Sets of Positive Upper Density}, Proceedings of the American Mathematical Society, \textbf{94}(3) (1985), 371--376.

\bibitem{BL}V. Bergelson and A. Leibman, \textit{Polynomial extensions of Van der Waerden's and Szemer\'{e}di's theorems}, Journal of the American Mathematical Society, \textbf{9} (1996), 725--753.

\bibitem{guth}
L. Guth, A. Iosevich, Y.  Ou, and H. Wang, \textit{On Falconer’s distance set problem in the plane} Inventiones mathematicae, \textbf{219}(3) (2020), 779--830.


\bibitem{NH} N. Hegyv\'ari, 
\textit{On intersecting properties of partitions of integers},
Combinatorics Probability and Computing, \textbf{14}(3) (2005), 319--323.


\bibitem{HPP}
N.~Hegy\-v\'{a}ri, J.~Pach, and T. Pham,
\textit{Polynomial extensions of Raimi's theorem}, 
\href{https://arxiv.org/abs/2511.06650}{arXiv:2511.06650 [math.CO]}, (2025)


\bibitem{HM79} N. Hindman, \textit{Ultrafilters and combinatorial number theory}, Number theory, Carbondale 1979 (Proc. Southern Illinois Conf., Southern Illinois Univ., Carbondale, Ill., 1979),  pp. 119--184, Lecture Notes in Math., 751, Springer, Berlin, 1979.


\bibitem{IR}
A. Iosevich and M. Rudnev, \textit{Erd\H{o}s distance problem in vector spaces over finite fields},  Transactions of the American Mathematical Society, \textbf{359}(12) (2007), 6127--6142.

\bibitem{Mat23} P. Mattila, \textit{A survey on the Hausdorff dimension of intersections},  Mathematical and Computational Applications, \textbf{28}(2) (2023), 49.


\bibitem{MPPRS}
B. Murphy, G. Petridis, T. Pham, M. Rudnev, and S. Stevens, \textit{On the pinned distances problem over finite fields},  Journal of the London Mathematical Society, \textbf{105}(1) (2022), 469--499.


\bibitem{PY}
T. Pham and S. Yoo, \textit{Intersection patterns and connections to distance problems}, 
\href{https://arxiv.org/abs/2304.08004}
{arXiv:2304.08004 [math.CO]}, (2023)


\bibitem{Raimi}  R. Raimi, \textit{Translation properties of finite partitions of the positive integers}, Fundamenta Mathematicae, \textbf{61}(3) (1967), 253--256.



\end{thebibliography}
\end{document}